\documentclass[reqno]{amsart}

\usepackage{amsmath,amssymb,amsthm,enumerate,mathrsfs}

\usepackage{hyperref}

\usepackage{color}

\numberwithin{equation}{section}

\theoremstyle{plain}
\newtheorem{Thm}{Theorem}[section]
\newtheorem{Prop}[Thm]{Proposition}
\newtheorem{Lem}[Thm]{Lemma}
\newtheorem{Cor}[Thm]{Corollary}

\theoremstyle{definition}
\newtheorem{Def}[Thm]{Definition}
\newtheorem{Expl}[Thm]{Example}

\theoremstyle{remark}
\newtheorem*{Rem}{Remark}

\renewcommand{\rho}{\varrho}

\title[The Cartan-Hadamard Theorem]{
The Cartan-Hadamard Theorem for Metric Spaces with Local Geodesic Bicombings
}

\author{
Benjamin Miesch}
\address{Department of Mathematics, ETH Z\"urich, 8092 Z\"urich, Switzerland}
\email{benjamin.miesch@math.ethz.ch}

\date{\today}

\begin{document}


\begin{abstract}
	Local-to-global principles are spread all-around in mathematics. The classical Cartan-Hadamard Theorem from Riemannian geometry was generalized by W. Ballmann for metric spaces with non-positive curvature, and by S. Alexander and R. Bishop for locally convex metric spaces.

In this paper, we prove the Cartan-Hadamard Theorem in a more general setting, namely for spaces which are not uniquely geodesic but locally possess a suitable selection of geodesics, a so-called convex geodesic bicombing.

	Furthermore, we deduce a local-to-global theorem for injective (or hyperconvex) metric spaces, saying that under certain conditions a complete, simply-connected, locally injective metric space is injective. A related result for absolute $1$-Lipschitz retracts follows.
\end{abstract}

\maketitle


\section{Introduction}\label{introduction}

	
	In a metric space $X$, a \emph{geodesic bicombing} is a selection of a geodesic between each pair of points. This is a map $\sigma \colon X \times X \times [0,1] \to X$ such that, for all $x,y \in X$, the path $\sigma_{xy}:=\sigma(x,y,\cdot)$ is a geodesic from $x$ to $y$. Moreover, we assume that this choice is consistent in the sense that $\sigma_{pq}([0,1]) \subset \sigma_{xy}([0,1])$ for all $p,q \in \sigma_{xy}([0,1])$ with $d(x,p) \leq d(x,q)$.
	A geodesic bicombing $\sigma$ is called \emph{convex} if the function $t \mapsto d(\sigma_{xy}(t),\sigma_{\bar{x}\bar{y}}(t))$ is convex for all $x,y,\bar{x},\bar{y} \in X$. Furthermore, we say that $\sigma$ is \emph{reversible} if $\sigma_{yx}([0,1]) = \sigma_{xy}([0,1])$ for all $x,y \in X$.
	
	Spaces with convex geodesic bicombings were studied extensively by D. Descombes and U. Lang in \cite{Des3,Des,Des2} and also by G. Basso in \cite{Bas}, where they show that several results for CAT(0) and Busemann spaces carry over to spaces with convex geodesic bicombings. Here we will contribute to these studies by proving the following Cartan-Hadamard Theorem.	

\begin{Thm}\label{thm:convex bicombing}
	Let $X$ be a complete, simply-connected metric space with a convex local geodesic bicombing $\sigma$. Then the induced length metric on $X$ admits a unique convex geodesic bicombing $\tilde{\sigma}$ which is consistent with $\sigma$. As a consequence, $X$ is contractible. Moreover, if the local geodesic bicombing $\sigma$ is reversible, then $\tilde{\sigma}$ is reversible as well.
\end{Thm}

As we show in a subsequent paper joined with G. Basso \cite{BasM}, this leads to a uniqueness result for convex geodesic bicombings on convex subsets of certain Banach spaces.

Important examples of spaces with convex geodesic bicombings are given by injective metric spaces. A metric space $X$ is \emph{injective} if for all metric spaces $A,B$ with $A \subset B$ and every 1-Lipschitz map $f \colon A \to X$, there is a 1-Lipschitz extension $\bar{f}\colon B \to X$, i.e. $\bar{f}|_A=f$.
In fact, D. Descombes and U. Lang show in their work that every proper, injective metric space of finite combinatorial dimension admits a (unique) convex geodesic bicombing \cite[Theorem 1.2]{Des}. Such spaces occur, for instance, as injective hulls of hyperbolic groups \cite[Theorem 1.4]{Lan} and therefore, every hyperbolic group acts properly and cocompactly by isometries on a space with a convex geodesic bicombing \cite[Theorem 1.3]{Des}.

Recall that injective metric spaces are complete, geodesic and contractible. Now, knowing that under the above conditions injective metric spaces possess a convex geodesic bicombing, we deduce the following local-to-global theorem for injective metric spaces.

\begin{Thm}\label{thm:local injective metric spaces}
	Let $X$ be a complete, locally compact, simply-connected, locally injective length space with locally finite combinatorial dimension. Then $X$ is an injective metric space.
\end{Thm}

It is well known that injective metric spaces are the same as absolute $1$-Lipschitz retracts. For Lipschitz retracts, the weaker notion of absolute Lipschitz uniform neighborhood retracts is common; e.g. see \cite{Haj}. Absolute $1$-Lipschitz uniform neighborhood retracts are locally injective but the converse is not true as we will see in Example~\ref{expl:sphere}. In fact, it turns out that the following holds.

\begin{Thm}\label{thm:absolute 1-lip nbhd retracts}
	Let $X$ be a locally compact absolute $1$-Lipschitz uniform neighborhood retract with locally finite combinatorial dimension. Then $X$ is an absolute $1$-Lipschitz retract.
\end{Thm}

This paper is organized as follows. We start Section~\ref{sec:local bicombings} by studying spaces with local geodesic bicombings, establish an appropriate exponential map and finally prove Theorem~\ref{thm:convex bicombing}.
In Section~\ref{sec:locally injective}, we first show that every uniformly locally injective metric space with a reversible, convex geodesic bicombing is injective. Afterwards, we describe how to construct a reversible, convex local geodesic bicombing on locally injective metric spaces, which extends to a convex geodesic bicombing by Theorem~\ref{thm:convex bicombing}. Thereby we establish Theorem~\ref{thm:local injective metric spaces}.
In the final Section~\ref{sec:absolute 1-lipschitz neighborhood retracts}, we then investigate absolute $1$-Lipschitz neighborhood retracts and prove Theorem~\ref{thm:absolute 1-lip nbhd retracts}.


\section{Local Geodesic Bicombings} \label{sec:local bicombings}

	
	Let us first fix some notation. In a metric space $X$, we denote by $$U(x,r) := \{ y \in X : d(x,y) < r \}$$ the open ball of radius $r$ around $x \in X$ and by $$B(x,r) := \{ y \in X : d(x,y) \leq r \}$$ the closed one.

	Let $X$ be a metric space and $\gamma \colon [0,1] \to X$ a continuous curve. The \emph{length} of $\gamma$ is given by
$$L(\gamma) := \sup \left\{ \sum_{k=1}^n d(\gamma(t_{k-1}),\gamma(t_k)) : 0 = t_0 < \ldots < t_n=1 \right\}.$$
Then
$$\bar{d}(x,y) := \inf \left\{ L(\gamma) : \gamma \colon [0,1] \to X, \gamma(0)=x, \gamma(1) = y \right\}$$ defines a metric on $X$, called the \emph{induced length metric}. If we have $d = \bar{d}$, we say that $(X,d)$ is a \emph{length space}.

	For a metric space $X$, let $\mathcal{G}(X) := \{ c \colon [0,1] \to X \}$ be the set of all geodesics in $X$, i.e. continuous maps $c \colon [0,1] \to X$ with $d(c(s),c(t)) = |s-t| \cdot d(c(0),c(1))$ for all $s,t \in [0,1]$. We equip $\mathcal{G}(X)$ with the metric $$D(c,c') := \sup_{t \in [0,1]} d(c(t),c'(t)).$$
	Let $c \in \mathcal{G}(X)$ and $0 \leq a \leq b \leq 1$, then the restriction $c|_{[a,b]}$ denotes the reparametrized geodesic given by $c|_{[a,b]} \colon [0,1] \to X$ with $c|_{[a,b]}(t) = c((1-t)a + tb)$.

\begin{Def}\label{def:local bicombing}
	 A \emph{local geodesic bicombing} on a metric space $X$ is a local selection of geodesics $\sigma \colon U \subset X \times X \to \mathcal{G}(X), (x,y) \mapsto \sigma_{xy}$ with the following properties:
	\begin{enumerate}[(i)]
		\item For all $x \in X$, there is some $r_x > 0$ such that, for all $y,z \in U(x,r_x)$, there is a geodesic $\sigma_{yz} \colon [0,1] \to U(x,r_x)$ from $y$ to $z$, and $$U = \{ (y,z) \in X \times X : y,z \in U(x,r_x) \text{ for some } x \}.$$
		\item The selection is consistent with taking subsegments of geodesics, i.e.
		$$\sigma_{\sigma_{xy}(a)\sigma_{xy}(b)}(t) = \sigma_{xy}((1-t)a+tb
)$$
	for $(x,y) \in U$, $0 \leq a \leq b \leq 1$ and $t \in [0,1]$.
	\end{enumerate}
\end{Def}

	We call a local geodesic bicombing $\sigma$ \emph{convex} if it is locally convex, i.e. for $y,z,y',z' \in U(x,r_x)$, it holds that
	$$t \mapsto d(\sigma_{yz}(t),\sigma_{y'z'}(t))$$
	is a convex function. Furthermore, $\sigma$ is \emph{reversible} if
	$$\sigma_{zy}(t) = \sigma_{yz}(1-t)$$
	for all $(y,z) \in U$ and $t \in [0,1]$.

\begin{Rem}
	Observe that, by consistency, a (local) geodesic bicombing is convex if and only if $$d(\sigma_{yz}(t),\sigma_{y'z'}(t)) \leq (1-t)d(y,y')+td(z,z')$$
	for all $(y,z),(y',z') \in U(x,r_x)$ and $t \in [0,1]$.
\end{Rem}

	A (local) geodesic $c \colon [0,1] \to X$ is \emph{consistent} with the local geodesic bicombing $\sigma$ if $$c ((1-t)a+tb) = \sigma_{c(a)c(b)}(t)$$ for all $0 \leq a \leq b \leq 1$ with $(c(a),c(b)) \in U$.

To prove Theorem~\ref{thm:convex bicombing}, we roughly follow the structure of Chapter~II.4 in \cite{Bri}. Adapting the methods of S. Alexander and R. Bishop \cite{Ale}, we can prove the following key lemma.

\begin{Lem}\label{lem:near geodesics}
	Let $X$ be a complete metric space with a convex local geodesic bicombing $\sigma$ and let $c$ be a local geodesic from $x$ to $y$ which is consistent with $\sigma$. Then, there is some $\epsilon > 0$ such that, for all $\bar{x},\bar{y} \in X$ with $d(x,\bar{x}),d(y,\bar{y}) < \epsilon$, there is a unique local geodesic $\bar{c}$ from $\bar{x}$ to $\bar{y}$ with $D(c,\bar{c}) < \epsilon$ which is consistent with $\sigma$. Moreover, we have $$L(\bar{c}) \leq L(c) + d(x,\bar{x}) + d(y,\bar{y})$$ and if $\tilde{c}$ is another consistent geodesic from $\tilde{x}$ to $\tilde{y}$ with $D(c,\tilde{c}) < \epsilon$, then $$t \mapsto d(\tilde{c}(t),\bar{c}(t))$$ is convex.
\end{Lem}

\begin{proof}
	Let $\epsilon > 0$ be such that $\sigma \big|_{ U(c(t), 2 \epsilon) \times U(c(t), 2 \epsilon) \times [0,1]}$ is a convex geodesic bicombing for all $t \in [0,1]$. Now, let $\mathsf{P}(A)$ be the following statement:
\begin{tabbing}
	$\mathsf{P}(A)$: \= For all $a,b \in [0,1]$ with $0 \leq b-a \leq A$ and for all $p,q \in X$ with \\
	\> $d(c(a),p), d(c(b),q) < \epsilon,$ there is a unique local geodesic $\bar{c}_{pq} \colon [0,1] \to X$ \\
	\> from $p$ to $q$ with $D(c|_{[a,b]},\bar{c}_{pq}) < \epsilon$ which is consistent with $\sigma$. Moreover, \\
	\> for all such local geodesics the map $t \mapsto d(\bar{c}_{pq}(t),\bar{c}_{p'q'}(t))$ is convex.
\end{tabbing}

	By our assumption, $\mathsf{P}(\frac{\epsilon}{l(c)})$ holds. Therefore, let us show $\mathsf{P}(A) \Rightarrow \mathsf{P}(\frac{3A}{2})$.
	
	Given $a,b \in [0,1]$ with $0 \leq b-a \leq \frac{3A}{2}$, define $p_0:=c(\frac{2}{3}a +\frac{1}{3}b)$ and $q_0:=c(\frac{1}{3}a +\frac{2}{3}b)$. Then, by $\mathsf{P}(A)$, there are consistent local geodesics $c_1$ from $p$ to $q_0$ and $c_1'$ from $p_0$ to $q$. Inductively, we set $p_n:=c_n(\frac{1}{2})$ and $q_n:=c_n'(\frac{1}{2})$, where $c_n$ is a consistent local geodesic from $p$ to $q_{n-1}$ and $c_n'$ from $p_{n-1}$ to $q$. Observe that, by convexity of the $c_n,c_n'$, we have $d(p_{n-1},p_n),d(q_{n-1},q_n) < \frac{\epsilon}{2^n}$ and hence the sequences $(p_n)_n$ and $(q_n)_n$ converge to some $p_\infty$ and $q_\infty$, respectively, and we have $d(p_\infty,p_0),d(q_\infty,q_0) < \epsilon$. Furthermore, by convexity, the $c_n,c_n'$ converge to the consistent local geodesics $c_\infty$ from $p$ to $q_\infty$ and $c_\infty'$ from $p_\infty$ to $q$, which coincide between $p_\infty=c_\infty(\frac{1}{2})$ and $q_\infty=c_\infty'(\frac{1}{2})$. Hence, they define a new local geodesic $c_{pq}$ from $p$ to $q$ which is consistent with $\sigma$ and $p_\infty=c_{pq}(\frac{1}{3})$, $q_\infty=c_{pq}(\frac{2}{3})$.
	
	Now, given two local geodesics $c_{pq}$ and $c_{p'q'}$ with $D(c|_{[a,b]},c_{pq}) < \epsilon$ and $D(c|_{[a,b]},c_{p'q'}) < \epsilon$,  set $s:=d(p,p')$, $t:=d(q,q')$, $s':=d(c_{pq}(\frac{1}{3}),c_{p'q'}(\frac{1}{3}))$ and $t':=d(c_{pq}(\frac{2}{3}),c_{p'q'}(\frac{2}{3}))$. Then we have $s' \leq \frac{s+t'}{2}$, $t' \leq \frac{s'+t}{2}$ and therefore $s' \leq \frac{s}{2}+ \frac{s'}{4} + \frac{t}{4}$, i.e. $s' \leq \frac{2s+t}{3}$ and similarly  $t' \leq \frac{s+2t}{3}$ follows. Hence, we get convexity of $t \mapsto d(c_{pq}(t),c_{p'q'}(t))$ and therefore also uniqueness follows.
	
	It remains to prove that $L(\bar{c}) \leq L(c) + d(x,\bar{x}) + d(y,\bar{y})$. Let $\tilde{c}$ be the unique consistent local geodesic from $x$ to $\bar{y}$ with $D(c,\tilde{c}) < \epsilon$. For $t$ small enough we have
	\begin{align*}
		t L(\tilde{c}) &= d(\tilde{c}(0),\tilde{c}(t)) = d(c(0),\tilde{c}(t)) \\
		&\leq d(c(0),c(t)) + d(c(t),\tilde{c}(t)) \leq t L(c) + t d(c(1),\tilde{c}(1)),
	\end{align*}
	i.e. $L(\tilde{c}) \leq L(c) + d(y,\bar{y})$ and similarly $L(\bar{c}) \leq L(\tilde{c}) + d(x,\bar{x})$.
\end{proof}

\begin{Def}\label{def:exp}
	Let $X$ be a metric space with a local geodesic bicombing $\sigma$. For some fixed $x_0 \in X$, we define
	$$\tilde{X}_{x_0} := \{  c \colon [0,1] \to X \text{ local geodesic with } c(0)=x_0, \text{ consistent with } \sigma \}.$$
	We equip $\tilde{X}_{x_0}$ with the metric $D(c,c') = \sup_{t \in [0,1]} d(c(t),c'(t))$ and define the map $$\exp \colon \tilde{X}_{x_0} \to X, \ c \mapsto c(1).$$
\end{Def}

If $X$ is complete, this map has the following properties.

\begin{Lem}\label{lem:properties of exp}
	Let $X$ be a complete metric space with a convex local geodesic bicombing $\sigma$. Then the following holds:
	\begin{enumerate}[(i)]
	\item The map $\exp \colon \tilde{X}_{x_0} \to X$ is locally an isometry. Hence $\sigma$ naturally induces a convex local geodesic bicombing $\tilde{\sigma}$ on $\tilde{X}_{x_0}$.
	
	\item $\tilde{X}_{x_0}$ is contractible.	
	
	\item For each $\tilde{x} \in \tilde{X}_{x_0}$, there is a unique local geodesic from $\tilde{x}_0$ to $\tilde{x}$ which is consistent with $\tilde{\sigma}$, where $\tilde{x}_0$ is the constant path $\tilde{x}_0(t)=x_0$.

	\item $\tilde{X}$ is complete.
	\end{enumerate}
\end{Lem}

\begin{proof}
	By Lemma~\ref{lem:near geodesics}, for every $c \in \tilde{X}_{x_0}$, there is some $\epsilon >0$ such that the map $\exp \big|_{U(c,\epsilon)} \colon U(c,\epsilon) \to U(c(1),\epsilon)$ is an isometry. Hence, $\sigma$ naturally induces a convex local geodesic bicombing $\tilde{\sigma}$ on $\tilde{X}_{x_0}$.

	Consider the map $r \colon \tilde{X}_{x_0} \times [0,1] \to \tilde{X}_{x_0}, (c,s) \mapsto (r_s(c) \colon t \mapsto c(st))$. This defines a retraction of $\tilde{X}_{x_0}$ to $\tilde{x}_0$.

	A continuous path $\tilde{c} \colon [0,1] \to \tilde{X}_{x_0}$ is a local geodesic in $\tilde{X}_{x_0}$ which is consistent with $\tilde{\sigma}$ if and only if $\exp \circ \tilde{c}$ is a local geodesic in $X$ which is consistent with $\sigma$. Therefore, for any $c \in \tilde{X}_{x_0}$, the map $s \mapsto r_s(c)$ is the unique local geodesic from $\tilde{x}_0$ to $c$.
	
	Finally, if $(c_n)_n$ is a Cauchy sequence in $\tilde{X}$, by completeness of $X$, for every $t \in [0,1]$, the sequences $(c_n(t))_n$ converge in $X$, to $c(t)$ say. Locally, i.e. inside $U(c(t),r_{c(t)})$, the subsegment $c|_{[t-\epsilon,t+\epsilon]}$ is the limit of the consistent geodesics $(c_n|_{[t-\epsilon,t+\epsilon]})_n$ and hence $c$ is consistent with $\sigma$ by the convexity of the local geodesic bicombing.
\end{proof}

The following criterion will ensure that $\exp$ is a covering map.

\begin{Lem}\label{lem:covering map}
	Let $p \colon \tilde{X} \to X$ be a map of length spaces such that
	\begin{enumerate}[(i)]
	\item \label{it:covering map i} $X$ is connected,
	
	\item \label{it:covering map ii} $p$ is a local homeomorphism,
	
	\item \label{it:covering map iii} for all rectifiable curves $\tilde{c}\colon [0,1] \to \tilde{X}$, we have $L(\tilde{c}) \leq L(p \circ \tilde{c})$,
	
	\item \label{it:covering map iv} $X$ has a convex local geodesic bicombing $\sigma$, and
	
	\item \label{it:covering map v} $\tilde{X}$ is complete.
	\end{enumerate}
	Then $p$ is a covering map.
\end{Lem}

\begin{proof}
	The proof of Proposition I.3.28 in \cite{Bri} also works in our setting. In the second step, take $U=U(x,r_x)$ and define the maps $s_{\tilde{x}}\colon U(x,r_x) \to \tilde{X}$ by $s_{\tilde{x}}(y) = \tilde{\sigma}_{xy}(1)$, where $\tilde{\sigma}_{xy}$ is the unique lift of $\sigma_{xy}$ with $\tilde{\sigma}_{xy}(0)=\tilde{x}$.
\end{proof}

\begin{Rem}
	For a local isometry $p$, conditions \eqref{it:covering map ii} and \eqref{it:covering map iii} are satisfied.
\end{Rem}

\begin{Cor}\label{cor:covering map}
	Let $(X,d)$ be a complete, connected metric space with a convex local geodesic bicombing $\sigma$. Then $\exp \colon \tilde{X}_{x_0} \to X$ is a universal covering map.
\end{Cor}

\begin{proof}
	Consider the induced length metrics $\bar{d}$ and $\bar{D}$ on $X$ and $\tilde{X}_{x_0}$. Since $(X,d)$ locally is a length space, the metrics $d$ and $D$ locally coincide with $\bar{d}$ and $\bar{D}$, respectively. Hence $p$ still is a local isometry with respect to the length metrics and $\sigma$ is a convex local geodesic bicombing. Thus Lemma~\ref{lem:covering map} applies.
\end{proof}

\begin{proof}[Proof of Theorem~\ref{thm:convex bicombing}.]
	First, we show that, for all $x,y \in X$, there is a unique consistent local geodesic from $x$ to $y$. Since $X$ is simply-connected, the covering map $\exp \colon \tilde{X}_{x} \to X$ is a homeomorphism which is a local isometry and by Lemma ~\ref{lem:properties of exp}, there is a unique consistent local geodesic $\tilde{\sigma}_{xy}$ from $x$ to $y$.

	Next, we prove that $\tilde{\sigma}_{xy}$ is a geodesic. To do so, it is enough to show that, for every curve $\gamma \colon [0,1] \to X$ and every $t \in [0,1]$, we have $L(\tilde{\sigma}_{\gamma(0)\gamma(t)}) \leq L(\gamma|_{[0,t]})$. Let
	$$A:= \Big\{ s \in [0,1] : \forall t \in [0,s] \text{ we have } L(\tilde{\sigma}_{\gamma(0)\gamma(t)}) \leq L(\gamma|_{[0,t]}) \Big\}.$$
	Clearly, $A$ is non-empty and closed. To prove that $A$ is open, consider $s \in A$. For $\delta>0$ small enough, by Lemma~\ref{lem:near geodesics}, we have
	\begin{align*}
		L(\tilde{\sigma}_{\gamma(0)\gamma(s+\delta)}) &\leq L(\tilde{\sigma}_{\gamma(0)\gamma(s)}) + d(\gamma(s),\gamma(s+\delta)) \\
		&\leq L(\gamma|_{[0,s]})+L(\gamma|_{[s,s+\delta]}) = L(\gamma|_{[0,s+\delta]}).
	\end{align*}
	Hence, $A=[0,1]$ as desired.
		
	Finally, we show that $t \mapsto d(\tilde{\sigma}_{xy}(t),\tilde{\sigma}_{\bar{x}\bar{y}}(t))$ is convex. By Lemma~\ref{lem:near geodesics}, there is a sequence $0=t_1 < \ldots < t_n=1$ and $\epsilon_k > 0$ such that 
	\begin{itemize}
	\item the balls $U(\tilde{\sigma}_{x \bar{x}}(t_1), \epsilon_1), \ldots, U(\tilde{\sigma}_{x \bar{x}}(t_n), \epsilon_n)$ cover $\tilde{\sigma}_{x\bar{x}}$,
	
	\item the balls $U(\tilde{\sigma}_{y\bar{y}}(t_1), \epsilon_1), \ldots, U(\tilde{\sigma}_{y\bar{y}}(t_n), \epsilon_n)$ cover $\tilde{\sigma}_{y\bar{y}}$, and
	
	\item for all $p,\bar{p} \in U(\tilde{\sigma}_{x\bar{x}}(t_k),\epsilon_k)$ and $q,\bar{q} \in U(\tilde{\sigma}_{y\bar{y}}(t_k),\epsilon_k)$, the map $t \mapsto d(\tilde{\sigma}_{pq}(t),\tilde{\sigma}_{\bar{p}\bar{q}}(t))$ is convex.
	\end{itemize}
	 Consider now a sequence $0 = s_0 < s_1 < \ldots < s_n = 1$ with
	 $$\tilde{\sigma}_{x\bar{x}}(s_k) \in U(\tilde{\sigma}_{x\bar{x}}(t_k), \epsilon_k)\cap U(\tilde{\sigma}_{x\bar{x}}(t_{k+1}), \epsilon_{k+1}),$$
	 $$\tilde{\sigma}_{y\bar{y}}(s_k) \in U(\tilde{\sigma}_{y\bar{y}}(t_k), \epsilon_k)\cap U(\tilde{\sigma}_{y\bar{y}}(t_{k+1}), \epsilon_{k+1}),$$
	 for $k=1, \ldots , n-1$. Then we get
	\begin{align*}
		d(\tilde{\sigma}_{xy}(t)&,\tilde{\sigma}_{\bar{x}\bar{y}}(t)) \\
		&\leq \sum_{k=1}^n d(\tilde{\sigma}_{\tilde{\sigma}_{x\bar{x}}(s_{k-1})\tilde{\sigma}_{y\bar{y}}(s_{k-1})}(t),\tilde{\sigma}_{\tilde{\sigma}_{x\bar{x}}(s_k)\tilde{\sigma}_{y\bar{y}}(s_k)}(t)) \\
		&\leq (1-t) \left(\sum_{k=1}^n d(\tilde{\sigma}_{x\bar{x}}(s_{k-1}),\tilde{\sigma}_{x\bar{x}}(s_k)) \right) + t \left(\sum_{k=1}^n d(\tilde{\sigma}_{y\bar{y}}(s_{k-1}),\tilde{\sigma}_{y\bar{y}}(s_k)) \right) \\
		&= (1-t) d(x,\bar{x}) + t d(y,\bar{y}).
	\end{align*}
	Hence, $\tilde{\sigma}$ is a convex geodesic bicombing on $X$.
	
	If $\sigma$ is reversible, then $\tilde{\sigma}^\ast_{xy}(t) := \tilde{\sigma}_{yx}(1-t)$ also defines a convex geodesic bicombing on $X$ which is consistent with $\sigma$. Therefore, by uniqueness, $\tilde{\sigma}^\ast$ and $\tilde{\sigma}$ coincide, i.e. $\tilde{\sigma}$ is reversible.
\end{proof}


\section{Locally Injective Metric Spaces} \label{sec:locally injective}


N. Aronszajn and P. Panitchpakdi \cite{Aro} proved that injective metric spaces are exactly the \emph{hyperconvex} metric spaces, namely metric spaces with the property that for every family of closed balls $\{ B(x_i,r_i) \}_{i \in I}$ with $d(x_i,x_j) \leq r_i + r_j$, for all $i,j \in I$, we have $\bigcap_{i \in I} B(x_i,r_i) \neq \emptyset$. Note that in hyperconvex metric spaces closed balls are hyperconvex.

\begin{Def}\label{def:locally injective}
	A metric space $X$ is \emph{locally injective} if, for every $x \in X$, there is some $r_x > 0$ such that $B(x,r_x)$ is injective.
	If we can take $r_x=r$ for all $x$ we call $X$ \emph{uniformly locally injective}.
\end{Def}

\begin{Lem}\label{lem:injective balls}
	Let $X$ be a metric space with the property that every closed ball $B(x,r)$ is injective, then $X$ is itself injective.
\end{Lem}

\begin{proof}
	Let $\{ B(x_i,r_i) \}_{i\in I}$ be a family of closed balls with $d(x_i,x_j) \leq r_i + r_j$. Fix some $i_0 \in I$ and set $A_i := B(x_i,r_i) \cap B(x_{i_0},r_{i_0})$. Since, for $r$ big enough, we have $x_i,x_j \in B(x_{i_0},r)$, we get that the $A_i$'s are externally hyperconvex in $A_{i_0}$ and $A_i \cap A_j \neq \emptyset$ for all $i,j \in I$. Hence, it follows
$$\bigcap_{i \in I} B(x_i,r_i) = \bigcap_{i \in I} A_i \neq \emptyset$$
by \cite[Proposition 1.2]{Mie}.
\end{proof}

\begin{Prop}\label{prop:uniformly locally injective with bicombing}
	Let $X$ be a uniformly locally injective metric space with a reversible, convex geodesic bicombing $\sigma$. Then $X$ is injective.
\end{Prop}

\begin{proof}
		Consider the following property:
\begin{tabbing}
$\mathsf{P}(R)$: \= For every family $\{ B(x_i,r_i) \}_{i\in I}$ with $d(x_i,x_j) \leq r_i+r_j$ and $r_i \leq R$, there is \\
	\> some $x \in \bigcap_{i\in I} B(x_i,r_i)$.
\end{tabbing} 
	
	Since $X$ is uniformly locally injective, this clearly holds for some $R_0 > 0$. Next, we show $\mathsf{P}(R) \Rightarrow \mathsf{P}(2R)$ and therefore $\mathsf{P}(R)$ holds for any $R \geq 0$.
	
	Let $\{ B(x_i,r_i) \}_{i\in I}$ be a family of closed balls with $d(x_i,x_j) \leq r_i+r_j$ and $r_i \leq 2R$. For $i,j \in I$, define $y_{ij} := \sigma_{x_i x_j}(\frac{1}{2})$. By convexity of $\sigma$, we have $$d(y_{ij},y_{ik}) = d(\sigma_{x_i x_j}(\tfrac{1}{2}),\sigma_{x_i x_k}(\tfrac{1}{2})) \leq \tfrac{1}{2}d(x_j,x_k) \leq \tfrac{r_j}{2}+ \tfrac{r_k}{2}.$$
	Hence, for every $i\in I$, there is some $z_i \in \bigcap_{j \in I} B(y_{ij}, \frac{r_j}{2})$. Now, observe that $d(z_i,z_j) \leq d(z_i,y_{ij}) + d(y_{ij},z_j) \leq \frac{r_i}{2}+\frac{r_j}{2}$ and therefore, we find
	$$x \in \bigcap_{i \in I} B(z_i,\tfrac{r_i}{2}) \subset \bigcap_{i \in I} B(x_i,r_i).$$
	
	Since all balls with center in $B(x,r)$ and radius larger than $2r$ contain $B(x,r)$, $\mathsf{P}(R)$ for $R =2r$ implies that $B(x,r)$ is injective. Hence, by Lemma~\ref{lem:injective balls}, $X$ is injective.
\end{proof}

Since compact, locally injective metric spaces are always uniformly locally injective we conclude the following.

\begin{Cor}\label{cor:compact locally injective with bicombing}
	Let $X$ be a compact, locally injective metric space with a reversible, convex geodesic bicombing $\sigma$. Then $X$ is injective.
\end{Cor}

\begin{Cor}\label{cor:proper locally injective with bicombing}
	Let $X$ be a proper, locally injective metric space with a reversible, convex geodesic bicombing $\sigma$. Then $X$ is injective.
\end{Cor}

\begin{proof}
	Let $\{ B(x_i,r_i) \}_{i\in I}$ be a family of balls with $d(x_i,x_j) \leq r_i+r_j$. Fix some $i_0 \in I$ and define $I_n = \{ i \in I : d(x_i,x_{i_0}) \leq n \}$, for $n \in \mathbb{N}$. Since $B(x_{i_0},n)$ is compact, by the previous corollary, there is some $y_n \in \bigcap_{i \in I_n} B(x_i,r_i)$. Especially, $(y_n)_n \subset B(x_{i_0},r_{i_0})$ and hence, there is some converging subsequence $y_{n_k} \to y \in \bigcap_{i \in I} B(x_i,r_i)$. 
\end{proof}

\begin{Rem}
In \cite{Lan}, U. Lang proves that every injective metric space admits a reversible, conical geodesic bicombing (Proposition 3.8). Observe also that this is the only property of the geodesic bicombing used in the proof of Proposition~\ref{prop:uniformly locally injective with bicombing}. Therefore, we get the following equivalence statement (in the terminology of \cite{Lan}): A metric space is injective if and only if it is uniformly locally injective and admits a reversible, conical geodesic bicombing.
\end{Rem}

If an injective metric space $X$ is proper, it also admits a (possibly non-consistent) convex geodesic bicombing \cite[Theorem 1.1]{Des} and if $X$ has finite combinatorial dimension in the sense of A. Dress \cite{Dre}, this convex geodesic bicombing is consistent, reversible and unique \cite[Theorem 1.2]{Des}. In our terms, this is:

\begin{Prop}\label{prop:unique convex bicombing}
	Every proper, injective metric space with finite combinatorial dimension admits a unique reversible, convex geodesic bicombing.
\end{Prop}

Recall that, by the Hopf-Rinow Theorem, any complete, locally compact length space is proper.

\begin{Cor}\label{cor:convex local bicombing}
	Let $X$ be a locally compact, locally injective metric space with locally finite combinatorial dimension. Then $X$ admits a reversible, convex local geodesic bicombing.
\end{Cor}

\begin{proof}
	For every $x \in X$, there is some $r_x > 0$ such that $B(x,3r_x)$ is compact, injective and has finite combinatorial dimension. This also holds for $B(x,r_x)$ and therefore, there is a reversible, convex geodesic bicombing $\sigma^x$ on $B(x,r_x)$.
	
	We will check that for $B(x,r_x)$ and $B(y,r_y)$ with $B(x,r_x)\cap B(y,r_y) \neq \emptyset$ the two geodesic bicombings $\sigma^x, \sigma^y$ coincide on the intersection. Assume without loss of generality that $r_x \geq r_y$ and hence $B(x,r_x),B(y,r_y) \subset B(x,3r_x)$. Then the convex geodesic bicombing $\tau$ on $B(x,3r_x)$ restricts to both $B(x,r_x)$ and $B(y,r_y)$ since, for $p,q \in B(z,r_z)$, we have
	$d(z,\tau_{pq}(t)) \leq (1-t)d(z,p) + t d(z,q) \leq r_z.$
	Hence, by uniqueness, the geodesic bicombings $\sigma^x, \sigma^y$ are both restrictions of $\tau$ and thus coincide on $B(x,r_x)\cap B(y,r_y)$.
	
	Therefore $\sigma$, defined by $\sigma|_{B(x,r_x) \times B(x,r_x)} := \sigma^x|_{B(x,r_x) \times B(x,r_x)}$, is a reversible, convex local geodesic bicombing on $X$.
\end{proof}

\begin{proof}[Proof of Theorem~\ref{thm:local injective metric spaces}.]
	Let $X$ be a complete, locally compact, simply-connected, locally injective length space with locally finite combinatorial dimension. By Corollary~\ref{cor:convex local bicombing}, $X$ has a reversible, convex local geodesic bicombing, which induces a reversible, convex geodesic bicombing by Theorem~\ref{thm:convex bicombing}. Hence, we can apply Corollary~\ref{cor:proper locally injective with bicombing} and deduce that $X$ is injective.
\end{proof}


\section{Absolute 1-Lipschitz Neighborhood Retracts} \label{sec:absolute 1-lipschitz neighborhood retracts}


A metric space $X$ is an \emph{absolute $1$-Lipschitz neighborhood retract} if, for every metric space $Y$ with $X \subset Y$, there is some neighborhood $U$ of $X$ in $Y$ and a $1$-Lipschitz retraction $\rho \colon U \to X$.
Furthermore, if we can take $U=U(X,\epsilon)$ for some $\epsilon > 0$, we call $X$ an \emph{absolute $1$-Lipschitz uniform neighborhood retract}. In this case, $\epsilon$ can be chosen independent of $Y$; see \cite[Proposition 7.78]{Haj}.

\begin{Lem}\label{lem:locally injective}
	Let $X$ be an absolute $1$-Lipschitz (uniform) neighborhood retract. Then $X$ is (uniformly) locally injective.
\end{Lem}

\begin{proof}
	Consider $X \subset l_\infty(X)$. Since $X$ is an absolute $1$-lipschitz neighborhood retract, there is some neighborhood $U$ of $X$ and a 1-Lipschitz retraction $\rho \colon U \to X$. For $x \in X$, there is some $r_x > 0$ such that $B(x,r_x) \subset U$. Let now $\{B(x_i,r_i)\}_{i \in I}$ be a family of closed balls with $x_i \in B(x,r_x)\cap X$ and $d(x_i,x_j) \leq r_i + r_j$. Then, since $l_\infty(X)$ is injective, there is some $y \in B(x,r_x) \cap \bigcap_{i \in I} B(x_i,r_i) \subset U$. Hence, we have $\rho(y) \in B(x,r_x) \cap \bigcap_{i \in I} B(x_i,r_i) \cap X$ and therefore $B(x,r_x) \cap X$ is injective.
	
	If $X$ is an absolute $1$-Lipschitz uniform neighborhood retract, we have $U=U(X,\epsilon)$ for some $\epsilon > 0$ and therefore, we can choose $r_x = \frac{\epsilon}{2}$ for all $x \in X$.
\end{proof}
The converse is not true, as the following example shows.

\begin{Expl}\label{expl:sphere}
	Consider the unit sphere $S^1$ endowed with the inner metric. Since, for every $x \in S^1$ and $\epsilon \in (0,\frac{\pi}{2}]$, the ball $B(x,\epsilon)$ is isometric to the interval $[-\epsilon,\epsilon]$, the unit sphere $S^1$ is uniformly locally injective.
	
	But $S^1$ is not an absolute $1$-Lipschitz neighborhood retract. Fix some inclusion $S^1 \subset l_\infty(S^1)$. We choose three points $x,y,z \in S^1$ with $r:=d(x,y)=d(x,z)=d(y,z)= \frac{2 \pi}{3}$. Let $U$ be a neighborhood of $S^1$ in $l_\infty(S^1)$. As $U$ is open, there is some $\epsilon \in (0,\frac{r}{2})$ such that $B(x,\epsilon) \subset U$. By hyperconvexity of $l_\infty(S^1)$, there is some $$p \in B(x,\epsilon) \cap B(y,r-\epsilon) \cap B(z,r-\epsilon) \subset U.$$ But since $$B(x,\epsilon) \cap B(y,r-\epsilon) \cap B(z,r-\epsilon)\cap S^1 = \emptyset,$$ there is no $1$-Lipschitz retraction $\rho \colon S^1 \cup \{p\} \to S^1$.
\end{Expl}

In fact, the notion of an absolute $1$-Lipschitz uniform neighborhood retract is quite restrictive.

\begin{Lem}\label{lem:geodesic}
	Let $X$ be an absolute $1$-Lipschitz uniform neighborhood retract. Then $X$ is
\begin{enumerate}[(i)]
	\item complete,
	\item geodesic, especially a length space, and
	\item simply-connected.
\end{enumerate}	
\end{Lem}

\begin{proof}
	Fix some inclusion $X\subset l_\infty(X)$ and $r=\frac{\epsilon}{2}>0$ such that there is a 1-Lipschitz retraction $\rho \colon U(X,\epsilon) \to X.$ 
	
	First, if $(x_n)_{n \in \mathbb{N}}$ is a Cauchy sequence in $X$, it converges to some $x \in U(X,\epsilon)$. It follows that $x=\rho(x) \in X$.
	
	Next, assume that there is a geodesic in $X$ between points at distance less than $d$. By Lemma~\ref{lem:locally injective}, this is clearly true for $d=r$. Consider two points $x,y \in X$ with $d(x,y) \leq d+r$. Now, since $l_\infty(X)$ is geodesic, there is some $z \in l_\infty(X)$ with $d(x,y) = d(x,z)+d(z,y)$, $d(x,z) \leq r$ and $d(z,y)\leq d$. But then, we have $\rho(z) \in X$ with $d(x,y) = d(x, \rho(z)) + d(\rho(z),y)$ and, by our hypothesis, there are geodesics from $x$ to $\rho(z)$ and from $\rho(z)$ to $y$ which combine to a geodesic from $x$ to $y$.
	
	Finally, since $X$ is locally simply-connected, every curve is homotopic to a curve of finite length and hence, it is enough to consider loops of finite length. We show that every such loop in $X$ is homotopic to a strictly shorter one and therefore, every loop is contractible.
	
	Let $\gamma$ be a loop in $X$ of length $L(\gamma) = 2 \pi R$ with $R > r$ and let $A = \{x \in \mathbb{R}^2 : R-r \leq \|x\| \leq R \}$ be the annulus bounded by the two circles $c = \{x \in \mathbb{R}^2 : \|x\| = R \}$ and $c' = \{x \in \mathbb{R}^2 : \|x\| = R-r \}$. Let $f$ be an isometry from $c$ onto $\gamma$ and $\bar{f} \colon A \to l_\infty(X)$ be a $1$-Lipschitz extension. Then $\gamma' = \rho \circ \bar{f}(c')$ is a loop of length $L(\gamma') \leq L(\gamma)-2 \pi r$ which is homotopic to $\gamma$. If $L(\gamma) \leq 2 \pi r$, we can use the same argument with $A$ replaced by the disk of perimeter $L(\gamma)$ to show that $\gamma$ is contractible. 
\end{proof}

We conclude that an absolute $1$-Lipschitz uniform neighborhood retract is a complete, simply-connected, locally injective length space and therefore Theorem~\ref{thm:absolute 1-lip nbhd retracts} follows.

\textbf{Acknowledgments.} I would like to thank Prof. Dr. Urs Lang for helpful remarks on this work and Dr. Ma\"el Pav\'on for inspiring discussions. I am also grateful for Giuliano Basso's comments related to Theorem~\ref{thm:convex bicombing}. The author was supported by the Swiss National Science Foundation.



\end{document}